\documentclass[11pt]{amsart}

\RequirePackage[l2tabu, orthodox]{nag}
\usepackage[T1]{fontenc}
\usepackage[utf8]{inputenc}
\usepackage{lmodern}
\usepackage[english]{babel}
\usepackage{amsrefs,amsthm,amssymb,amsmath}
\usepackage[pdftex]{graphicx}
\usepackage{enumerate}
\usepackage{url}
\usepackage[foot]{amsaddr}

\newcommand{\dfcn}[4]{\setlength{\arraycolsep}{1.5pt}\begin{array}{ccc}#1&\to&#2\\#3&\mapsto&#4\end{array}}

\newcommand{\R}{\mathbb{R}}

\newcommand{\Z}{\mathbb{Z}}
\newcommand{\T}{\mathbb{S}^1}
\DeclareMathOperator{\Homeo}{\mathrm{Homeo}}
\newcommand{\rot}{\mathsf{rot}}
\newcommand{\Fix}{\mathsf{Fix}}

\newtheorem{thm}{Theorem}
\newtheorem*{claim}{Claim}
\newtheorem{lem}[thm]{Lemma}
\newtheorem{prop}[thm]{Proposition}

\usepackage[margin=3cm]{geometry}

\usepackage{indentfirst}
\usepackage{microtype}

\usepackage[colorlinks=true,linkcolor=blue,citecolor=magenta]{hyperref}

\title[On James Hyde's example]{On James Hyde's example of non-orderable subgroup of $\mathrm{Homeo}(D,\partial D)$}
\date{}
\author{Michele Triestino}
\address{Institut de Math\'ematiques de Bourgogne, Universit\'e Bourgogne Franche-Comt\'e, CNRS UMR 5584,
9 av.~Alain Savary,
21000 Dijon, France.}
\email{michele.triestino@u-bourgogne.fr}

\begin{document}

\begin{abstract}
	In [Ann.\ Math.\ 190 (2019), 657--661], James Hyde presented the first example of non-left-orderable, finitely generated subgroup of $\mathrm{Homeo}(D,\partial D)$, the group of homeomorphisms of the disk fixing the boundary. This implies that the group $\mathrm{Homeo}(D,\partial D)$ itself is not left-orderable. We revisit the construction, and present a slightly different proof of purely dynamical flavor, avoiding direct references to properties of left-orders. Our approach allows to solve the analogue problem for actions on the circle.
	
	\smallskip
	{\noindent\footnotesize \textbf{MSC\textup{2010}:} Primary 37C85. Secondary 37E05, 37E10, 37E20.}
\end{abstract}

\maketitle

\section{Introduction}

Let $\Homeo(D,\partial D)$ denote the group of homeomorphisms of the disk $D$ which fix the boundary $\partial D$. In \cite{Hyde} James Hyde gave a bright proof of the fact that $\Homeo(D,\partial D)$ is non-left-orderable, solving a fundamental question about this group: it is different, even at algebraic level, from the group of homeomorphisms of the real line.
Indeed, it is a classical fact (attributed to Conrad \cite{conrad}) that a countable group is left-orderable if and only if it admits a faithful action on the real line by orientation-preserving homeomorphisms, and by Burns--Hale Theorem, a group is left-orderable if and only if all finitely generated subgroups are; see Clay and Rolfsen \cite{CR}.
 In these terms, the result of James Hyde can be stated as follows:

\begin{thm}[Hyde]\label{t:hyde}
There exists an explicit finitely generated subgroup $G$ of $ \Homeo(D,\partial D)$ which does not embed into the group of orientation-preserving homeomorphisms of the real line $\Homeo_+(\R)$.
\end{thm}

Confirming the above-mentioned idea that $\Homeo(D,\partial D)$ owns a ``higher-dimensional algebraic structure'', we extend the result of Hyde to actions on the circle.

\begin{thm}\label{t:circle}
	There exists an explicit finitely generated subgroup $G$ of $\Homeo(D,\partial D)$ which does not embed into the group of homeomorphisms of the circle $\Homeo(\T)$.
\end{thm}

This is done by rewriting the proof of Hyde in terms of one-dimensional actions, with minor changes. Instead of the nice algebraic bounds used in \cite{Hyde} (which would correspond to bounds on displacement functions of elements in the group), we use the commutation relations in the group defined by Hyde, and the classifications of actions of the Klein bottle group. Note also that the group $G$ that we consider for Theorem \ref{t:circle} is slightly different (but it contains Hyde's group as a subgroup).

\section{Preliminaries on actions on the real line}\label{s:line}

In the following, all actions on the real line will be assumed to be by orientation-preserving homeomorphisms, unless explicitly mentioned.
The reader who is familiar with groups acting on the real line can skip this part, as all results presented here are classical.

It will be important to have a precise picture of possible faithful actions on $\R$ of the groups 
\[\Z^2=\langle f,g\mid fgf^{-1}=g\rangle\quad\text{and}\quad K=\langle f,g\mid fgf^{-1}=g^{-1}\rangle.\]
The group $K$ is the classical \emph{Klein bottle group}.
Up to restrict the action to some invariant interval, we can assume that our actions have no global fixed points. In the following, we will write $f$ and $g$ for the generators of the groups $\Z^2$ and $K$, as above. 

Denote by $\Fix(g)=\{x\in \R\mid g(x)=x\}$ the set of fixed points of $g$, and observe that for any homeomorphisms $f$ and $g\in \Homeo_+(\R)$, one has $\Fix(g)=\Fix(g^{-1})$ and $\Fix(fgf^{-1})=f(\Fix(g))$.
This implies that for any action of either $\Z^2$ or $K$ on the real line, the set of fixed points $\Fix(g)$ is preserved by~$f$.

\begin{lem}\label{l:fixf}
For any faithful action of either $\Z^2$ or $K$ on the real line without global fixed points, if $\Fix(g)\neq\varnothing$ then $\Fix(f)=\varnothing$.

In other terms, for any action of either $\Z^2$ or $K$ on the real line, if both generators $f$ and $g$ have fixed points, then the action has a global fixed point.
\end{lem}
\begin{proof}
Assume by contradiction $\Fix(f)\neq\varnothing$. Given $x\in \Fix(g)$, then $\{f^n(x)\}_{n\in\Z}$ accumulates on some point $p\in \Fix(f)$. Observe that by $f$-invariance of $\Fix(g)$, we have $\{f^n(x)\}\subset \Fix(g)$. Moreover $\Fix(g)$ is closed, so we must also have $p\in \Fix(g)$. This gives a point $p$ which is fixed by both $f$ and $g$, and hence by the whole group. This gives the desired contradiction.
\end{proof}

In the case of the group $K$, the condition $\Fix(g)\neq\varnothing$ is always satisfied.

\begin{lem}\label{l:fixg}
For any action of $K$ on the real line, we always have $\Fix(g)\neq\varnothing$. Moreover, one has the inclusion $\Fix(f)\subset \Fix(g)$.
\end{lem}
\begin{proof}
We assume for contradiction $\Fix(g)=\varnothing$, and without loss of generality we can assume $g(x)>x$ for every $x\in \R$ (otherwise we consider the inverse $g^{-1}$). Therefore $gf^{-1}(x)>f^{-1}(x)$ for every $x\in \R$. As $f$ preserves orientation, this implies $fgf^{-1}(x)>x$ for every $x\in \R$, and consequently the relation $fgf^{-1}=g^{-1}$ implies $g^{-1}(x)>x$, which is in contradiction with our assumption.
The second assertion follows from the proof of Lemma~\ref{l:fixf}.
\end{proof}

We can now describe all possible faithful actions of either $\Z^2$ or $K$ on the real line without global fixed points, with the condition $\Fix(g)\neq \varnothing$. This is not strictly needed for the rest of the text, but it helps the reader to make a picture of the dynamics under consideration.

\begin{lem}\label{l:actions}
Consider an action of  either $\Z^2$ or $K$ on the real line without global fixed points, with the condition $\Fix(g)\neq \varnothing$. Then the action is $C^0$ conjugate to an action obtained from the construction below.

Assume $f(x)=x+1$ or $f(x)=x-1$.
Given $x\in \R$, consider the interval $I$ joining the points $x$ and $f(x)$. Given any orientation-preserving homeomorphism $h:I\to I$, there exists a unique orientation-preserving homeomorphism $g$ which satisfies $g\vert_I=h$ and $fgf^{-1}=g^{\epsilon}$, where $\epsilon\in\{-1,+1\}$ is chosen accordingly to the group that is acting ($\epsilon =1$ in case of $\Z^2$ and $\epsilon=-1$ in case of $K$).
\end{lem}

\begin{proof}
The fact that $g$ is uniquely defined by the homeomorphism $h:I\to I$ is because one must have
\[
g\vert_{f^n(I)}=f^nh^{\epsilon^n}f^{-n}\vert_{f^n(I)}
\]
after the relation in the group.

Reversely, by Lemma \ref{l:fixf}, we have that $f$ has no fixed point and thus $f$ is $C^0$ conjugate to either $x\mapsto x+1$ or $x\mapsto x-1$. Take then any $x\in \Fix(g)$, and observe that also $f(x)\in \Fix(g)$, so that the restriction $g\vert_I$ defines an orientation-preserving homeomorphism of $I$, where $I$ is the interval joining $x$ and $f(x)$.
\end{proof}

\section{Preliminaries on actions on the circle}\label{s:p-circle}
We keep the assumption that actions preserve the orientation.
Given an orientation-preserving homeomorphism $f:\T\to \T$, we denote by $\rot(f)\in \R/\Z$ its rotation number, which equals
\[
\rot(f)=\mu[0,f(0)),
\]
for any $f$-invariant Borel probability measure $\mu$.
We will use the following classical facts (see e.g.\ Ghys \cite{Ghys}):
\begin{itemize}
\item $\rot(f)=0$ if and only if $f$ has a fixed point;
\item the rotation number is a conjugacy invariant (that is, $\rot(hfh^{-1})=\rot(h)$);
\item for any amenable group $G$, $\rot:G\to \R/\Z$ defines a group homomorphism (in particular, $\rot(f^n)=n\,\rot(f)$).
\end{itemize} 
Given an action of $K=\langle f,g\mid fgf^{-1}=g^{-1}\rangle$ on the circle, the conjugacy-invariance of the rotation number gives
\[
\rot(g)\in \left \{0,\tfrac12\right \},
\]
therefore $g^2$ always has a fixed point. We will need an improved version of this fact:

\begin{lem}\label{l:kleincircle}
	Consider a faithful action of $K$ on the circle, with $\Fix(f)\neq \varnothing$. Then $\Fix(f)\subset\Fix(g^2)$. As a consequence, for any connected component $I$ of $\T\setminus\Fix(f)$ one has that $\Fix(g^2)\cap I$ is infinite.
\end{lem}

\begin{proof}
	By assumption, both $f$ and $g^2$ have fixed points. Observe that the subgroup $K_0$ of $K$ generated by $f$ and $g^2$ is isomorphic to $K$. In particular it is amenable, and it preserves a Borel probability measure on $\T$, whose support is contained in $\Fix(f)\cap \Fix(g^2)$, which is thus nonempty. As a consequence, we have that the subgroup $K_0$ generated by $f$ and $g^2$ acts on the circle $\T$ with a global fixed point. Therefore the action of $K_0$ reduces to an action on the real line (identified with the complement of a global fixed point), to which we can apply Lemma \ref{l:fixg}.
\end{proof}

\section{Choice of generators}
Consider the following one-parameter families of planar homeomorphisms:
\begin{align*}
a^t(x,y)&=\left (x+t,y\right ),\\
b^t(x,y)& = \left (x,y+t\right ),\\
c^t(x,y)&=\left (x,y+t\gamma_0(x)\right ),
\end{align*}
where $t\in \R$ and  $\gamma_0:\R\to \R$ is the $1$-periodic function such that
\begin{equation}\label{eq:gamma}
\gamma_0(x)=\left \{\begin{array}{lr}
-4x+1& \text{if }x\in [0,1/2),\\[.5em]
4x-3 & \text{if }x\in [1/2,1).
\end{array}\right.
\end{equation}
We fix
\[
\alpha=a^{1/12},\quad \beta=b^{1/12},\quad \gamma=c^{1/168}.
\]
We also introduce the planar homeomorphism
\[
\delta(x,y)=(\delta_0(x),y),
\]
where $\delta_0:\R\to\R$ is the orientation-preserving homeomorphism of the real line such that $\delta_0(x+1/2)=\delta_0^{-1}(x)+1/2$ and
\[
\delta_0(x)=\left \{\begin{array}{lr}
\frac12 x& \text{if }x\in [0,1/3),\\[.5em]
2x-\frac12 & \text{if }x\in [1/3,1/2).
\end{array}\right.
\]
Observe that we also have the relation $\gamma_0(x+1/2)=-\gamma_0(x)$. See Figure \ref{f:delta0}  for the graphs of these two functions.

We let $H=\langle \alpha,\beta,\gamma,\delta\rangle$ denote the subgroup of $\Homeo(\R^2)$ generated by these four homeomorphisms.

\begin{figure}[ht]
\[\includegraphics[scale=1]{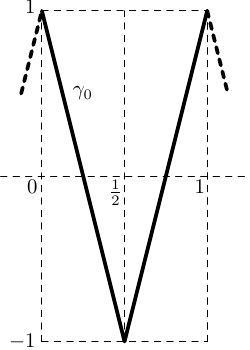}
\quad
\includegraphics[scale=1]{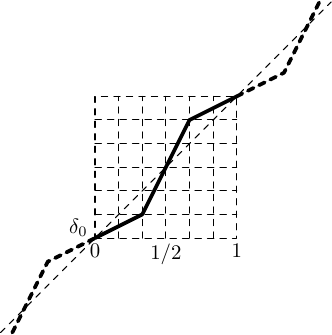}\]
\caption{Graphs of $\gamma_0$ (left) and $\delta_0$ (right).}\label{f:delta0}
\end{figure}

\section{Properties}
It is clear that the four generators of $H$ display bounded displacement. More precisely, one checks that
\[
\|s(x,y)-(x,y)\|\le 1/3\quad\text{for every }(x,y)\in \R^2\text{ and }s\in\{\alpha,\beta,\gamma,\delta\}.
\]
Choosing then the identification
\[
\dfcn{\R^2}{[-\pi/2,\pi/2]^2}{(x,y)}{(\arctan x,\arctan y)}\]
we can see the four homeomorphisms as elements of the group $\Homeo(I^2,\partial I^2)$ of homeomorphisms of the square which fix the boundary (which is clearly isomorphic to $\Homeo(D,\partial D)$).

We first compute
\begin{equation}\label{eq:delta2}
\delta_0^2(x)=\left\{
\begin{array}{lr}
\tfrac14x&\text{if }x\in [0,4/12),\\[.5em]
x-\tfrac14&\text{if }x\in[4/12,5/12),\\[.5em]
4x-\tfrac32&\text{if }x\in[5/12,1/2).
\end{array}
\right.
\end{equation}
We record the following basic but important properties:
\begin{itemize}
\item $\beta$ commutes with every other homeomorphism $\alpha,\gamma$, and $\delta$,
\item $\alpha^6\gamma \alpha^{-6}=\gamma^{-1}$,
\item $\alpha^6\delta \alpha^{-6}=\delta^{-1}$,
\item all the elements $g_k:=\alpha^k\delta^{-2}\gamma\delta^2\alpha^{-k}$ (for $k\in \Z$) pairwise commute.
\end{itemize}
Observe that $\alpha^{12}$ commutes with both $\delta$ and $\gamma$, so we have $g_k=g_{k+12}$ for every $k\in \Z$.
We also have the following key relation:
\begin{lem}\label{lem:relation}
We have the relation
\[
\prod_{k=0}^{11}g_k = b^{1/24}
\]
and thus
\[
\prod_{k=0}^{11}g^2_k = \beta.
\]
\end{lem}

\begin{proof}
We set $t=1/168$.
We first check that
\[
g_0(x,y)=\delta^{-2} \gamma \delta^2(x,y)= \left (x,y+t \gamma_0\left (\delta^2_0(x)\right )\right )
\]
and therefore
\[
g_k(x,y)=\alpha^kg_0\alpha^{-k}(x,y)=\left (x,y+t\gamma_0\left (\delta_0^2\left (x-\tfrac{k}{12}\right )\right )\right ).
\]
We deduce
\begin{equation}\label{eq:key}
\prod_{k=0}^{11}g_k = \left (x,y+t\sum_{k=0}^{11}\gamma_0\left (\delta_0^2\left (x-\tfrac{k}{12}\right )\right )\right ).
\end{equation}
\begin{claim}
For every $x\in \R$, we have
\[
\sum_{k=0}^{11}\gamma_0\left (\delta_0^2\left (x-\tfrac{k}{12}\right )\right )=7.
\]
\end{claim}
\begin{proof}[Proof of Claim]
Note that the function $\varphi(x)=\sum_{k=0}^{11}\gamma_0\left (\delta^2_0\left (x-\tfrac{k}{12}\right )\right )$ is $1/12$-periodic and differentiable at every $x\in (0,1/12)$.
Given $x\in (0,1/12)$, we write $x_k=x-\tfrac{k}{12}$ and $y_k=\delta^2_0(x_k)$. Then
\[
\varphi'(x)=\sum_{k=0}^{11} \gamma_0'(y_k)(\delta^2_0)'(x_k).
\]
Observe that:
\begin{itemize}
	\item the derivative of $\gamma_0$ is constant on both intervals $(0,1/2)$ and $(1/2,1)$, which are preserved by $\delta_0^2$;
	\item the derivative of $\delta_0^2$ is constant on any interval of the form $\left (\tfrac{k}{12},\tfrac{k+1}{12}\right )$, $k\in \Z$.
\end{itemize}
Then, by close inspection of the values of these derivatives, by means of the expressions \eqref{eq:gamma} and \eqref{eq:delta2}, one finds that
\[
\varphi'(x)=-4\left (4\cdot \tfrac14+1+4\right )+4\left (4+1+4\cdot\tfrac14\right )=0.
\]
Therefore $\varphi$ is constant and it is enough to compute its value at the point $0$:
\[
\varphi(0)=2\left (\left (\sum_{k=1}^{4}\left (-\tfrac{k}{12}+1\right )\right )-4\cdot \tfrac{5}{12}+2\right )=7\]
(this can be done by evaluating the composition $\gamma_0\circ \delta_0^2$ from the expressions \eqref{eq:gamma} and \eqref{eq:delta2} of $\gamma_0$ and $\delta_0^2$, respectively).
\end{proof}
After the claim and the expression \eqref{eq:key}, we have
\[
\prod_{k=0}^{11}g_k = (x,y+7t)=b^{7t}=b^{7/168}=b^{1/24}.\qedhere
\]
\end{proof}

\section{On actions on the real line}
\begin{prop}\label{p:fixbeta}
For any faithful action of $H$ on $\R$ and connected component $I$ of $\R\setminus \Fix(\alpha)$, we have that $\Fix(\beta)\cap I$ is infinite.
\end{prop}

\begin{proof}
We let $I\subset \R$ be a connected component of $\R\setminus\Fix(\alpha)$. Observe that $\Fix(\alpha^k)=\Fix(\alpha)$ for every $k\neq 0$.
The relations $\alpha^6 \gamma \alpha^{-6}=\gamma^{-1}$ and $\alpha^6 \delta \alpha^{-6}=\delta^{-1}$ imply (by Lemma \ref{l:fixg}) that both $\gamma$ and $\delta$ preserve $I$ and have fixed points inside $I$. This also implies that for every $k\in \Z$, the element $g_k=\alpha^k\delta^{-2}\gamma\delta^2\alpha^{-k}$ preserves $I$ and has fixed points in $I$.
 We have already observed that the subgroup $A=\langle g_0,\ldots,g_{11}\rangle\le H$ is abelian.  Applying Lemma \ref{l:fixf} to the action of $A$ on $I\cong \R$, by an inductive argument on the rank, we deduce that $A$ admits a global fixed point in $I$. By Lemma \ref{lem:relation}, we have $\beta\in A$, so that $\beta$ fixes a point in $I$. As $\alpha$ and $\beta$ commute, we actually see that $\beta$ has infinitely many fixed points in $I$. 
\end{proof}

\section{On actions on the circle}
\begin{lem}\label{l:fixbetacircle}
	For any faithful, orientation-preserving action of $H$ on $\T$, one has $\Fix(\beta)\neq\varnothing$.
\end{lem}
\begin{proof}
As explained in Section \ref{s:p-circle}, the relation $\alpha^6\gamma\alpha^{-6}=\gamma^{-1}$ implies that $\gamma^2$ has a fixed point, and so does every conjugate $g^2_k$. In terms of rotation number, this gives $\rot(g^2_k)=0$. Using Lemma \ref{lem:relation}, we get $\rot(\beta)=\sum_{k=0}^{11}\rot(g_k^2)=0$ (recall that the function rotation number is a homomorphism in restriction to amenable groups). We conclude that $\beta$ has a fixed point.
\end{proof}

\begin{prop}\label{p:fixbetacircle}
	Consider any faithful action of $H$ on $\T$ with $\Fix(\alpha)\neq\varnothing$. For any connected component $I$ of $\T\setminus \Fix(\alpha)$, we have that $\Fix(\beta)\cap I$ is infinite.
\end{prop}

\begin{proof}
Given a faithful action of $H$ on the circle, we deduce from the relations $\alpha^6\gamma\alpha^{-6}=\gamma^{-1}$ and Lemma \ref{l:kleincircle} that $\Fix(\alpha)\subset\Fix(\gamma^2)\cap \Fix(\delta^2)$. Therefore, every $g^2_k$ preserves any connected component $I$ of $\T\setminus \Fix(\alpha)$, and we conclude as in Proposition \ref{p:fixbeta}, considering the abelian subgroup $\langle g_0^2,\ldots,g_{11}^2\rangle$, which contains $\beta$.
\end{proof}

\section{Involution and conclusion}

Next, let $\eta(x,y)=(y,x)$ be the involution exchanging the two axes, and define $\overline{\gamma}=\eta \gamma \eta$ and $\overline \delta=\eta\delta\eta$.
Observe that $\eta$ normalizes $\Homeo(I^2,\partial I^2)$, and that $\eta\alpha\eta=\beta$.

We shall first give a proof of Theorem \ref{t:hyde}, although it is a formal consequence of Theorem \ref{t:circle}.

\begin{proof}[Proof of Theorem \ref{t:hyde}]
Let $H=\langle \alpha,\beta,\gamma,\delta\rangle$ be as above.
Let $I\subset\R\setminus \Fix(\alpha)$ be a connected component. After Proposition \ref{p:fixbeta}, there exists a connected component $J$ of $\R\setminus \Fix(\beta)$ which is contained in $I$.
On the other hand, applying Proposition \ref{p:fixbeta} to the group $\overline H=\langle \alpha, \beta, \overline \gamma,\overline \delta\rangle$, we deduce that for any faithful action of $\overline H$ on the real line, for every connected component $J$ of the complement of $\Fix(\beta)$ in $\R$, we must have that $\Fix(\alpha)\cap J$ is infinite.
From this, we conclude that there can be no faithful action of the group $G=\langle H,\overline H\rangle$ on $\R$. Observe that by construction, we have $G\le \Homeo(I^2,\partial I^2)$. This proves Theorem \ref{t:hyde}.
\end{proof}

\begin{proof}[Proof of Theorem \ref{t:circle}]
	First of all, we note that for Theorem \ref{t:circle}, it suffices to find a finitely generated subgroup $G$ which does not act on the circle by \emph{orientation-preserving} homeomorphisms.
	Indeed, all the generators $\alpha,\beta,\gamma,\delta,\overline \gamma,\overline \delta\in \Homeo(\R^2)$ admit a ``square root'' in $\Homeo(D,\partial D)$, that is, an element whose square gives the generator. It is then enough to work with the group $\widetilde G$ generated by the square roots, because for every homeomorphism $f:\T\to \T$, the square $f^2$ preserves the orientation, so that for any representation $\rho:\widetilde G\to \Homeo(\T)$, we will have the inclusion $\rho(G)\subset \Homeo_+(\T)$.

Now, for any faithful action of $G$ on the circle, applying Lemma \ref{l:fixbetacircle} to both $H$ and $\overline{H}$, we deduce that both $\beta$ and $\alpha$ have fixed points. Thus we can apply Proposition \ref{p:fixbetacircle} and conclude as for the proof of Theorem \ref{t:hyde}.
\end{proof}

\section*{Acknowledgments}
We thank Ioannis  Iakovoglou and Luis Paris for the encouragement to write down a detailed proof. The author  has been partially supported by the projects ANR Gromeov (ANR-19-CE40-0007) and ANER Agroupes (AAP 2019 R\'egion Bourgogne--Franche--Comt\'e).


\begin{bibdiv}
\begin{biblist}


\bib{CR}{book}{
	author = {\scshape Clay, A.},
	author = {\scshape Rolfsen, D.},
   title={Ordered groups and topology},
   series={Graduate Studies in Mathematics},
   volume={176},
   publisher={American Mathematical Society, Providence, RI},
   date={2016},
   pages={x+154},
   isbn={978-1-4704-3106-8},
   review={\MR{3560661}},
}

\bib{conrad}{article}{
	author={Conrad, P. F.},
	title={Embedding theorems for abelian groups with valuations},
	journal={Amer. J. Math.},
	volume={75},
	date={1953},
	pages={1--29},
	issn={0002-9327},
	review={\MR{53933}},
}


\bib{Ghys}{article}{
	AUTHOR = {\scshape Ghys, \'{E}.},
     TITLE = {Groups acting on the circle},
   JOURNAL = {Enseign. Math. (2)},
     VOLUME = {47},
      YEAR = {2001},
    NUMBER = {3-4},
     PAGES = {329--407},
     review={\MR{1876932}},
}

\bib{Hyde}{article}{
	author = {\scshape Hyde, J.},
	title = {The group of boundary fixing homeomorphisms of the disc is not left-orderable},
	journal = {Ann. Math. (2)},
	volume = {190},
	number = {2},
	pages = {657--661},
	year = {2019},
	issn={0003-486X},
   review={\MR{3997131}},
}


\end{biblist}
\end{bibdiv}
\end{document}